\documentclass{amsart}

\usepackage{amsthm}
\usepackage{hyperref}

\DeclareMathOperator{\coh}{\mathrm{H}}

\DeclareMathOperator{\GU}{\mathrm{GU}}
\DeclareMathOperator{\End}{\mathrm{End}}
\DeclareMathOperator{\Nm}{\mathrm{Nm}}
\DeclareMathOperator{\NS}{\mathrm{NS}}
\DeclareMathOperator{\Pic}{\mathrm{Pic}}
\DeclareMathOperator{\SU}{\mathrm{SU}}
\DeclareMathOperator{\su}{\mathrm{su}}

\DeclareMathOperator{\SL}{\mathrm{SL}}
\DeclareMathOperator{\disc}{\mathrm{disc}}
\DeclareMathOperator{\sig}{\mathrm{sig}}

\theoremstyle{plain}
\newtheorem{proposition}[equation]{Proposition}

\newtheorem{theorem}[equation]{Theorem}

\newtheorem{lemma}[equation]{Lemma}

\theoremstyle{definition}
\newtheorem{definition}[equation]{Definition}

\theoremstyle{remark}

\newcommand\calS{\mathcal{S}}
\newcommand\Sbar{\bar{\calS}}

\title[Discriminants]{Discriminants of Hermitian forms, maximal degenerations and Kontsevich's tropical approach to the Hodge conjecture}
\author[Brosnan]{Patrick Brosnan}
\address{Department of Mathematics\\
  University of Maryland\\
  College Park, MD USA}
\email{pbrosnan@umd.edu}

\begin{document}
\begin{abstract}
Motivated by Markman's work and by Kontsevich's tropical approach to disproving
the Hodge conjecture, 
this note explains why families of Weil type abelian $2n$-folds only 
have maximal unipotent degeneration when the discriminant is $(-1)^n$.
Along the way, it explains what the discriminant is and collects 
various facts about Weil type abelian varieties.
\end{abstract}

\maketitle

\section{Introduction}
\label{intro}
The main goal of this note is to point out the connection between the
discriminant of a Hermitian form $H$ of signature $(n,n)$ over an imaginary
quadratic field $\mathbb{Q}(\sqrt{-d})$ and the possible degenerations of
families of Weil type abelian varieties with associated  Hermitian form $H$.  
As I show in Proposition~\ref{p.main}, maximally degenerate families can occur
only when the discriminant of $H$ is $(-1)^n$.

The motivation for making this observation has to do with two
relatively recent developments: 
(i) M.~Kontsevich's tropical approach
to the Hodge conjecture~\cite{zharkov2020tropical} and 
(ii) E.~Markman's proof~\cite{markmanweilhodge} of the Hodge
conjecture for Weil type cycles on Weil abelian $4$-folds of
discriminant along with his more recent
proof~\cite{markmanCyclesAbelian2nfolds2025} of the Hodge conjecture for all abelian $4$-folds and for Weil
type abelian $6$-folds of discriminant $-1$. 

(i) is a program initiated by Kontsevich to look for counterexamples
to the usual Hodge conjecture via a related, but, by nature, more
combinatorial, and, therefore, hopefully more amenable to computation,
tropical Hodge conjecture.  
I.~Zharkov's preprint~\cite{zharkov2020tropical} writes down some of
the details of this program and how it might have 
conceivably worked in the
context of Weil type abelian $4$-folds. 

As stated in the
first paragraph of~\cite{zharkov2020tropical}, Kontsevich's
program requires a maximal degeneration of $g$-dimensional abelian varieties to
get a $g$-dimensional tropical abelian variety. 
However, by Proposition~\ref{p.main}, the only moduli spaces of Weil type
abelian varieties which admit such
maximal degenerations are those with discriminant $(-1)^{g/2}$. 
(Note that Weil type abelian varieties are always even dimensional.)

In April of 2023, when the first version of this note was written, 
Markman had proved the Hodge conjecture for Weil type abelian 
$4$-folds of discriminant $1$~\cite{markmanweilhodge}.  
So, by the above argument, Proposition~\ref{p.main} showed that 
Kontsevich's method could not disprove the Hodge conjecture for 
Weil type abelian $4$-folds.
Now, Markman has proved the Hodge conjecture for \emph{all} abelian
$4$-folds~
\cite{markmanCyclesAbelian2nfolds2025}.
Given this new development, this note has nothing interesting to say for
abelian $4$-folds.
However, also in~\cite{markmanCyclesAbelian2nfolds2025}, Markman
proved the Hodge conjecture for Weil type abelian $6$-folds of
discriminant $-1$.
The results of this note then imply that Kontsevich's method
cannot disprove the Hodge conjecture for Weil type abelian $6$-folds.

It is worth  pointing out
explicitly that Zharkov's preprint~\cite{zharkov2020tropical} does not, in
fact, claim to produce a counterexample to the Hodge conjecture.
Rather, the preprint ends inconclusively with what is essentially an
unsuccessful computer computation aimed at disproving a tropical form of the
Hodge conjecture.

Most of the work done in writing this note involved collecting facts about
(a) Weil type abelian varieties and (b) Hermitian forms. 
In \S\ref{s.weiltype}, I summarize results about Weil type abelian varieties, mostly
taken from a paper of van Geemen~\cite{vanGeemen94hodge}.
In~\S\ref{s.shimura}, I  summarize more results from van Geemen's paper
about the Shimura varieties arising as moduli spaces of Weil type abelian
varieties.
As it turned out, I did not really wind up needing that section
anywhere. 
But I want to keep it because I think it is conceptually useful. 
Although it is not the point of view taken in this note, 
probably the most efficient way to think about the results of this
note is in terms of Shimura varieties associated to groups $G$
with $G_{\mathbb{R}} = \GU(n,n)$, where $\GU(n,n)$ is the 
general unitary group associated to a Hermitian form $H$ of 
signature $(n,n)$. 
From that point of view, the point is that $\GU(n,n)$ is 
quasisplit if and only if the descriminant of $H$ is $(-1)^n$.
Moreover, $\GU(n,n)$ being quasisplit is equivalent to the 
existence of a maximal unipotent degeneration.

In~\S\ref{s.main}, I state Proposition~\ref{p.main}, which is the main
observation of this note.
If you take $n=2$ in the Proposition, you get that Weil type abelian
$4$-folds of discriminant not equal to $1$ cannot have maximal degenerations
in the sense used by Zharkov in~\cite{zharkov2020tropical}.
Similarly, if you take $n=3$, you get that Weil type abelian $6$-folds
of discriminant $-1$ cannot have maximal unipotent degenerations.

Section~\ref{s.herm} collects several (very standard) facts about Hermitian
forms, many of which are used 
in the proof of Proposition~\ref{p.main}.  
In \S\ref{s.proof}, I  prove Proposition~\ref{p.main}.
Finally, I end with a short conclusion~\S\ref{s.conclusion}, where I point
out that, while you don't get maximal degenerations of the kind needed
in~\cite{zharkov2020tropical}, you do get degenerations that are essentially
``one less than maximal." 
(So, for example, you can always get a family of Weil type abelian $4$-folds
to degenerate to a $2$-dimensional tropical torus, even if they won't always
degenerate to a $4$-dimensional torus.)

\subsection{Acknowledgments}  
This note started out as a letter to Helge Ruddat about M.~Kontsevich's talk  
``Motives beyond geometry" given on April 25, 2023 
over Zoom for the Concluding Conference of Simons Collaboration on Homological Mirror Symmetry
at the Simons Institute.
The first version of the note was written at the conference itself.
This was well before Markman's recent work on the Hodge conjecture 
for Weil type abelian $6$ folds of disriminant $-1$.

I thank the Simons Institute, the organizers of the conference, and Kontsevich
for the interesting talk.
I also thank Helge Ruddat for the very interesting conversations we had
after the talk.
We had initially viewed this note as one part of a joint project, but eventually 
we decided to separate it out.
I thank N.~Fakhruddin, B.~ van
Geemen,  and I.~Zharkov for reading earlier versions of the note and responding with 
valuable suggestions and comments. 
I also thank E.~Markman, who was, incidentally, also at the Simons Institute
conference, for his help and encouragement both at the conference
itself and afterwards.

\section{Weil type Abelian varieties}
\label{s.weiltype}
We first need to recall a little bit
about Weil type Abelian varieties and the Shimura varieties that
parametrize them.
For this,  I follow van Geemen's paper~\cite{vanGeemen94hodge}.

Suppose $X$ is a complex abelian variety. 
Write $\End X$ for the ring of endomorphisms of $X$ as an algebraic group.
For each $\varphi\in\End X$, we get a linear transformation 
$d\varphi_0\in \End_{\mathbb{C}} T_0 X$, where $T_0 X$ is the fiber of the sheaf of holomorphic 
tangent vectors over the point $0\in X$. 
We wind up getting a ring homomorphism $t:\End(X)\otimes\mathbb{Q}\to \End_{\mathbb{C}} T_0 X$.

The data defining a Weil type abelian variety is a pair $(X,K)$, where $X$ is a  
$2n$-dimensional abelian variety and $K$ is a subalgebra 
of $\End(X)\otimes\mathbb{Q}$ isomorphic to an imaginary quadratic extension of $\mathbb{Q}$.
Fix an embedding $K\hookrightarrow\mathbb{C}$.
Then I say that $(X,K)$ is an \emph{abelian variety of Weil type} if, for each $\alpha\in K$,
the transformation $t(\alpha)$ is conjugate to the 
diagonal $2n\times 2n$ matrix with $n$ eigenvalues $\alpha$ and $n$ eigenvalues
$\bar\alpha$. (See~\cite[Definition
4.9]{vanGeemen94hodge}, and note that this definition does not depend on the choice of embedding 
of $K$ in $\mathbb{C}$.)

Suppose $(X,K)$ is a Weil type abelian variety.
Then for $\alpha\in K$, we get a map $\alpha^*:\coh^*(X,\mathbb{Q})\to \coh^*(X,\mathbb{Q})$.
The map $\alpha\mapsto\alpha^*$ is a homomorphism of multiplicative monoids,
and it gives $\coh^1(X,\mathbb{Q})$ the structure of a $K$-vector 
space~\cite[\S5.1]{vanGeemen94hodge}. 
(On the other hand, it's easy to see that it doesn't give all of $\coh^*(X,\mathbb{Q})$ a 
$K$-vector space structure: just think about $\coh^0(X,\mathbb{Q})$.)

Since $K$ is an imaginary quadratic extension of $\mathbb{Q}$, we have $K=\mathbb{Q}(\sqrt{-d})$ for some 
positive integer $d$. 
Write $\NS X$ for the N\'eron-Severi group of $X$.
This is the image of the map $\Pic X\to \coh^2(X,\mathbf{Z}(1))$ from the Picard group of line bundles
on $X$ to the cohomology of $X$ with coefficients in the group $\mathbb{Z}(1)=2\pi i\mathbb{Z}$ taking a line
bundle $\mathcal{L}$ to its first Chern class $c_1\mathcal{L}$.
A \emph{polarized Weil type abelian variety} is then a triple $(X,K,E)$ where $(X,K)$ is a Weil type
abelian variety and $E\in \NS X$ is the image of an ample line bundle with the property that 
\begin{equation}
  \label{e.pol}
  (\sqrt{-d})^*E = dE. 
\end{equation}
Note that, for any $\alpha\in\mathbb{Q}$, we have $(\alpha)^*=\alpha^i$ on $\coh^i(X,\mathbb{Q})$.
From this, it follows easily that the definition of polarization via \eqref{e.pol} is independent of the choice of $d$.

In \cite[Lemma 5.2]{vanGeemen94hodge}, van Geemen gives some important facts
about polarized Weil type abelian varieties $(X,K,E)$.
One crucial fact for what follows is the existence of a canonical Hermitian form 
\begin{equation}
  \label{e.Herm}
  H: \coh_1(X,\mathbb{Q})\times \coh_1(X,\mathbb{Q})\to K
\end{equation}
given by $H(x,y) := E(x,(\sqrt{-d})_*y) + \sqrt{-d}E(x,y)$.

Write $\SU_H$ for the special unitary group over $\mathbb{Q}$ associated to $H$,
and write $V:=\coh_1(X,\mathbb{Q})$. 
The set of $\mathbb{Q}$-valued points of  $\SU_H$ is 
\begin{equation}
  \label{e.suhq}
  \SU_H(\mathbb{Q})=\{T\in\SL(V):\text{for all $v,w\in V$, } H(Tv,Tw)=H(v,w)\}. 
\end{equation}

By~\cite[Lemma 5.2 (4)]{vanGeemen94hodge}, the form $H$ has signature $(n,n)$. 
In other words, if we base-change $H$ to a Hermitian form on $V\times_K \mathbb{C}\cong \mathbb{C}^{2n}$, then it is isometric
to the form $H(z, w) = \sum_{i=1}^n \bar z_i w_i - \sum_{i=n+1}^{2n} \bar z_i w_i$.
So 
the base-change $\SU_{H,\mathbb{R}}$ of $\SU_H$ to $\mathbb{R}$ is the group $SU(n,n)$ associated to 
Hermitian form on $\mathbb{C}^{2n}$ of signature $(n,n)$.  
Moreover,  by~\cite[Lemma 6.10]{vanGeemen94hodge}, we have
$\SU_{H,\mathbb{C}}\cong \SL_{2n,\mathbb{C}}$.
One compact way of saying this is to say that $\SU_{H,\mathbb{R}}$ is the
unique (up to isomorphism) quasi-split but not split simply connected algebraic
group over $\mathbb{R}$ of type $A_{2n-1}$. 

\section{Weil type PEL Shimura varieties}
\label{s.shimura}
Now, to get a maximal degeneration, we need a family of varieties.  
And, fortunately, the Weil type abelian varieties $(X,K)$ of dimension $2n$ are
parametrized by locally symmetric spaces of dimension $n^2$.
In other words, they live in families of dimension $n^2$. 

In~\cite[Section 5]{vanGeemen94hodge}, van Geemen explains this. 
The upshot is the following. 
Write $G=\SU_H$. 
Then, as above, $G_{\mathbb{R}}\cong SU(n,n)$, and there is a compact subgroup
$K\cong S(U(n)\times U(n))$ of $G(\mathbb{R})$, where $S(U(n)\times
U(n))=\{(A,B)\in U(n)\times U(n): \det AB=1\}$. 
If $\Lambda$ is a lattice in $V$, then write $\Gamma=\Gamma_{\Lambda}$ 
for the set of all $T\in \SU_{H}$ such that $T\Lambda =\Lambda$. 
(In~\cite{vanGeemen94hodge}, van Geemen writes ``$T\Lambda\subseteq\Lambda$'' in~\cite{vanGeemen94hodge}, but this is equivalent to the 
condition that $T\Lambda=\Lambda$, which makes it more obvious that $\Gamma$ is a subgroup.)
Then, as in \cite[\S5.11]{vanGeemen94hodge}, 
$\mathcal{H}:=\Gamma\backslash\SU_{H}(\mathbb{R})/K$ 
has the structure of a quasi-projective variety of dimension $n^2$, which parametrizes $2n$-dimensional abelian varieties of  
polarized Weil type.
Moreover, every polarized Weil type abelian variety appears in such a family $\mathcal{H}$ for some polarization 
$H$.
Informally, the space $\mathcal{H}$ is sometimes called a ``Shimura variety." 
(However, note that this does not agree with Deligne's official
definition of what a Shimura variety is~\cite[\S
2.1.2]{deligneVarietesShimuraInterpretation1979}.)

To avoid having too many things labeled by the same letter, let's write $\mathcal{S}(H,\Gamma):=
\Gamma\backslash\SU_{H}(\mathbb{R})/K$ for the quasi-projective variety $\mathcal{H}$ above.
If $\Gamma'\leq \Gamma$ is a finite index subgroup, then I get a finite morphism
$\mathcal{S}(H,\Gamma')\to \mathcal{S}(H,\Gamma)$.
Moreover, if $\Gamma$ is sufficiently small, there is a universal family 
of Weil type abelian $2n$-folds over $\mathcal{S}(H,\Gamma)$.
In particular, for $\Gamma$ sufficiently small, $\mathcal{S}(H,\Gamma)$ is a fine moduli space, and we get a universal family $\mathcal{A}\to\mathcal{S}(H,\Gamma)$
of Weil type polarized abelian varieties.

In fact, there's an explicit condition on $\Gamma$, which guarantees that $\mathcal{S}(H,\Gamma)$ is 
a fine moduli space. 
That is that $\Gamma$ be what's called neat.
In~\cite[\S17.1]{borarg}, Borel says what it means for a subgroup to be neat 
(or \emph{net} in French)
and, in~\cite[\S17.4]{borarg},  
he shows that sufficiently small arithmetic subgroups are neat.

What kind of degenerations the universal family $\mathcal{A}\to \mathcal{S}(H,\Gamma)$ can have is completely 
determined by the group $G=\SU_H$.
To explain this properly, I have to talk about the Baily-Borel compactification $\bar{\mathcal{S}}=\bar{\mathcal{S}}(H,\Gamma)$ of 
$\mathcal{S}=\mathcal{S}(H,\Gamma)$. 
The Baily-Borel compactification  $\Sbar$ is a normal projective algebraic variety containing 
$\calS$ as a dense open subset. 
It is stratified by locally closed subvarieties, and the strata are in one-one correspondence with 
$\Gamma$-conjugacy classes of maximal parabolic subgroups of $G$.
(To be completely explicit, the strata are in one-one correspondence with
$\Gamma$-conjugacy classes of parabolic subgroups $P<G$ which are defined over
$\mathbb{Q}$ and maximal among proper parabolic subgroups of $G$ defined over $\mathbb{Q}$.)

In~\S\ref{s.main} below, I'll analyze the degenerations of Weil type abelian varieties enough for the purposes of this 
note without using the structure of $\Sbar$.
But understanding $\Sbar$ would be useful for getting a more precise understanding of the degenerations that exist.

\section{The main observation}
\label{s.main}

\begin{proposition}
  \label{p.main}
  Suppose $K/\mathbb{Q}$ is an imaginary quadratic field, $V\cong K^{2n}$ 
  is an even-dimensional $K$-vector space and $H$ is a nondegenerate
  hermitian form on $V$.  Consider the following statements:
  \begin{enumerate}
    \item\label{lag} there is an $n$-dimensional $K$-linear subspace $W\subseteq V$ with 
      $H(x,y)=0$ for all $x,y\in W$. 
    \item\label{disc} the discriminant of $H$ is $(-1)^n$ and the signature of $H$ is $(n,n)$.
    \item\label{degen} there is a maximal degeneration $A\to B$ of Weil type abelian varieties 
      with hermitian form isometric to $H$.
  \end{enumerate}
  Then \eqref{degen} $\Rightarrow$ \eqref{lag} $\Leftrightarrow$ \eqref{disc}.
\end{proposition}

In fact, all of the conditions above are equivalent, but I don't need that to 
establish the main point of this note (that having a maximal degeneration
for Abelian $4$-folds of Weil type implies having discriminant $(-1)^n$).
I'll prove Proposition~\ref{p.main} in \S\ref{s.proof} below 
after explaining some points about Hermitian forms in \S\ref{s.herm}. 
But now, I want to be a little bit more explicit about what the items mean.

The meaning of \eqref{lag} is obvious.  
We'll call such a $W$ if it exists a \emph{lagrangian subspace} of $V$.

To explain what the discriminant of $H$ means in  
Proposition~\ref{p.main}~\eqref{disc}, pick a $K$-basis $v_1,\ldots, v_{2n}$ of
$V$, and use it to write $H$ as a Hermitian matrix $a_{ij}$.
Explicitly, $H(v_i,v_j)=a_{ij}$.
Write $A=(a_{ij})$. 
Then $\det A\in\mathbb{Q}^{\times}$, and it is well-defined (independent of the choice of basis) 
modulo $\Nm(K^{\times})$, where $\Nm:K\to\mathbb{Q}$ 
is the norm map $c\mapsto c\bar c$. 
By definition, the \emph{discriminant} of $H$ is $\det A\in \mathbb{Q}^{\times}/\Nm(K^{\times})$.
(See~\cite[\S 4.4]{vangeemen2022fourfolds}.)
Let's write $\disc H$ for the discriminant of $H$ or $\disc V$ for the discriminant of $H$ on $V$ if $H$ 
is fixed.

Finally, by a maximal degeneration, I mean a family $A\to B$ of Weil type abelian varieties  
such that the following conditions hold: 
\begin{enumerate}
  \item $B$ is a smooth curve included as a Zariski open subset in a smooth, connected, complex curve $\bar B$ with 
    $\bar B\setminus B =\{b_0\}$ a single closed point;
  \item $A\to B$ is smooth, projective and every fiber is a Weil type abelian variety;
  \item\label{isom} there is a Weil type polarization $E$ on the family such that the associated hermitian form $H_E$ on $\coh_1(A_b,\mathbb{Q})$
    for some (hence any) $b\in B$ is isometrically equivalent to $H$;
  \item\label{maxup} if we let $N$ denote the logarithm of the unipotent part of the monodromy of the family $A\to B$ about $b_0$ 
    acting on $\coh_1(A_b,\mathbb{Q})$ for some $b\in B$, then $N$ has $2n$ Jordan blocks of size $2$.
    Or equivalently, $N$ has rank $2n$ over $\mathbb{Q}$.
\end{enumerate}

It's really condition~\eqref{maxup} that has to do with the maximal unipotent conditions.
And this is the condition that is required in the first paragraph of Zharkov's paper 
for the real dimension of the tropical abelian variety to be equal to the complex dimension of
the Weil type abelian varieties~\cite{zharkov2020tropical}.
(The key phrase there is that ``half the $1$-cycles vanish.")

By the Monodromy Theorem~\cite[Theorem 6.1]{schmid73}, since we are talking about a
family of abelian varieties, we have $N^2=0$.
So all Jordan blocks of $N$ are either of dimension $1$ or $2$.
This explains why the condition on the Jordan blocks is equivalent to the condition on the rank of $N$ as a 
$\mathbb{Q}$-linear transformation.

The rest of the conditions are just there to 
ensure that we're talking about the right kind of varieties.  

With the exception of the definition of isometry of Hermitian forms, this
explains the meaning of the terms in Proposition~\ref{p.main} (3). 
I'll give the precise definition of isometry  in
\S\ref{s.herm} below.

\subsection{Landherr's theorem}
In the 1930s, Landherr~\cite{landherr1935formen} proved (a much more general
form of) the following theorem.

\begin{theorem}[Landherr]
  \label{t.landherr}
   A nondegenerate hermitian form over $K$ is determined up to isometry by the dimension of its underlying space, its signature
   and its discriminant.
\end{theorem}

So another way to express \eqref{isom} is to say that $H_E$ has signature $(n,n)$ and the same discriminant as $H$. 
I'll wind up using Landherr's theorem below. 
Although Landherr's paper is very short, it is written in an old-fashioned way without
theorem statements.  
(In other words, Landherr just goes and proves the theorem without ever explicitly stating it.)
So, it might help to point that Landherr's theorem is stated explicitly on page 262 of the 
D.~Lewis' survey article~\cite{lewis82isometry}. 
It is also stated (and even proved) in the examples near the end of
Jacobson's very short article~\cite{jacobson40hermitian}. 

\section{Hermitian forms} 
\label{s.herm}

In the course of proving Proposition~\ref{p.main}, it's going to be useful to
have a little bit of terminology about Hermitian forms over an imaginary quadratic
filed $K=\mathbb{Q}(\sqrt{-d})$.

First, let's call a pair $(X,\varphi_X)$ consisting of a finite-dimensional $K$-vector space 
$X$ and a non-degenerate hermitian form $\varphi_X$ a \emph{hermitian space}.
If $\varphi_X$ is fixed, it's often convenient to refer to the pair just by the 
letter $X$.  
So I'll do that sometimes.

We say two Hermitian spaces $(X,\varphi_X)$ and $(Y,\varphi_Y)$ are \emph{isometric}
if there is a $K$-linear isomorphism $T:X\to Y$ such that, for all $x_1,x_2\in X$,
$\varphi_Y(Tx_1,Tx_2)= \varphi_X(x_1,x_2)$.
Isometry is easily seen to be an equivalence relation 
(the inverse of an isometry
and the composition of isometries are isometries).
We write $X\cong Y$ if $X$ and $Y$ are isometric.

Suppose $X$ and $Y$ are $K$-vector spaces with non-degenerate hermitian forms $\varphi_X$
and $\varphi_Y$ respectively. 
We can put a non-degenerate Hermitian form $\varphi$ on $X\oplus Y$ by the formula
\begin{equation}
  \label{e.perp}
  \varphi((x_1,y_1),(x_2,y_2)) = \varphi_X(x_1,x_2) + \varphi_Y(y_1,y_2).
\end{equation}
This gives a new hermitian space $(X\oplus Y, \varphi)$, which I call 
the \emph{orthogonal direct sum} of $X$ and $Y$ and write as $X\perp Y$ for short. 
Note that $\disc X\perp Y = (\disc X)(\disc Y)$ and, if we write $\sig X\in\mathbb{Z}^2$
for the signature of $X$ (viewed as an ordered pair of integers), then 
$\sig (X\perp Y) = \sig X + \sig Y$. 

On the other hand, suppose $(V,\varphi)$ is a hermitian space and $X\subseteq V$
is a $K$-linear subspace such that $\varphi_X:=\varphi_{|X}$ is a non-degenerate
Hermitian form. 
Set $Y=X^{\perp} := \{v\in V: \varphi(v,x)=0$ for all $x\in X\}$. 
Then it's not hard to see that $\varphi_{|Y}$ is also non-degenerate and $V=X\oplus Y$.
Moreover, the canonical map $X\oplus Y \to V$ induces an isometry from 
$X\perp Y$ to $V$.
In this case, I write $V=X\perp Y$.

For any $a\in\mathbb{Q}\setminus\{0\}$, write $\langle a\rangle$ for the Hermitian form $\varphi$ on $K$ given by 
$\varphi(x,y)=a\bar x y$. 
Then write $\langle a_1,\ldots, a_n\rangle$ for the orthogonal sum  
$\langle a_1\rangle \perp \cdots \perp \langle a_n\rangle$.
Clearly, $\disc \langle a_1,\ldots, a_n\rangle = \prod_{i=1}^n a_i$, 
and $\sig \langle a_1,\ldots, a_n\rangle = (k, n-k)$ where $k=\#\{i: a_i>0\}$. 

Here's a lemma, which is, incidentally, proved (in greater generality) on the first page of Landherr's paper~\cite{landherr1935formen}.

\begin{lemma}
  \label{l.direct} 
  Suppose $X=(X,\varphi)$ is a non-degenerate Hermitian form over $K$ with $0<n=\dim X<\infty$. 
  Then, there exists $a_1,\dots, a_n\in\mathbb{Q}^{\times}$, such that $X$ is isometric to 
  $\langle a_1,\ldots, a_n\rangle$. 
\end{lemma}
\begin{proof}
   We claim first that we can find $x\in X$ with $\varphi(x,x)\neq 0$.
   To see this, suppose, to get a contradiction, that $\varphi(x,x)=0$ for all $x\in X$.
   Pick a nonzero $x\in X$. 
   Since $\varphi$ is non-degenerate, we can find $y\in X$ with $\varphi(x,y)=1$.
   But then $\varphi(x+y,x+y) =2 \neq 0$. 
   So we arrive at a contradiction.

   Now the lemma is proved by induction on dimension.
   Pick a nonzero $x\in X$ with $a:=\varphi(x,x)\neq 0$.
   (Since $\varphi(x,x)=\overline{\varphi(x,x)}$, $a\in\mathbb{Q}$.)
   Then $X=\langle a\rangle \perp Y$ with $Y=Kx^{\perp}$ and with $\varphi_{|Y}$ non-degenerate.
   Then, by induction, we're done.
\end{proof}

\begin{definition}
  \label{d.hyper} The \emph{hyperbolic plane} is the Hermitian form on $K^2$ given by  $\mathbb{H} := \langle 1, -1\rangle$.  
  For each natural number $k$, let $\mathbb{H}^k$ denote the $k$-fold orthogonal direct sum 
  $\mathbb{H}\perp \cdots \perp\mathbb{H}$ of $\mathbb{H}$.
\end{definition}

\begin{lemma}
  \label{l.class}
  Suppose $\varphi$ is a Hermitian form on $V=K^{2n}$ of signature $(n,n)$.
  Then $V\cong \mathbb{H}^{n-1}\perp X$ for some $X\cong K^2$ with signature $(1,1)$.
  Moreover, we have $\disc V = (-1)^n$ if and only if $V\cong \mathbb{H}^n$. 
\end{lemma}
\begin{proof}
   Since $\sig V=(n,n)$, we can write find positive rational numbers
   $a_1$, $\ldots$, $a_n$, $b_1$, $\ldots$, $b_n$  such that $V\cong \langle a_1,\ldots,
   a_n, -b_1,\ldots, -b_n\rangle$.
   Set $a=\prod a_i$ and $b=\prod b_i$.
   Then $\disc V = (-1)^n ab$.
   Set $X=\langle a, -b\rangle$. 
   Then $\disc \mathbb{H}^{n-1}\perp X = (-1)^{n-1} \disc X = (-1)^{n-1}(-ab)= (-1)^n ab = \disc V$.
   So $V$ and $\mathbb{H}^{n-1}\perp X$ have the same dimension, discriminant and signature.
   By Landherr (Theorem~\ref{t.landherr}), it follows that they are isometric.

   On the other hand, if $\disc V = (-1)^n$, then $V$ and $\mathbb{H}^{n}$ have the same dimension, discriminant and signature.
   So, by Landherr, $V\cong\mathbb{H}^n$.
   This proves the last statement of the lemma.
\end{proof}

We want to introduce one more concept that will help to streamline  the proof
of Proposition~\ref{p.main} below. 
If $(V,\varphi)$ is a hermitian space, then I say that a non-zero vector $w\in V$ is \emph{isotropic}
if $\varphi(w,w)=0$. 
For example, the vector $(1,e^{2\pi i \alpha})$ is isotropic in the hyperbolic plane $\mathbb{H}$ for 
any $\alpha\in\mathbb{R}$.
On the other hand, we have the following lemma (which is really the basis of Witt's theory in the case of quadratic forms).

\begin{lemma}
\label{l.iso}
   Suppose $(V,\varphi)$ is a non-degenerate hermitian space. 
   If $w\in V$ is a nonzero isotropic vector, then $V=X\perp X^{\perp}$ where $X$ is a $2$-dimensional subspace of $V$
   isometric to the hyperbolic plane $\mathbb{H}$ containing $w$.  
\end{lemma}
\begin{proof}
   Since $\varphi$ is non-degenerate, we can find $u\in V$ such that $\varphi(u,w)=1/2$.
   Let $X$ denote the $K$-span of $u$ and $w$.
   Setting, $v=u-(u,u)w$, we get that $\varphi(v,v) = (u,u) - 2 (u,u)(u,w) = 0$.
   So then $\varphi(v,v)=\varphi(w,w)=0$ and $\varphi(v,w)=1/2$.
   Now, set $x=v+w$, $y=v-w$.
   We get that $\varphi(x,x)=1$, $\varphi(y,y)=-1$ and $\varphi(x,y)=0$.
   So $X\cong\mathbb{H}$.
\end{proof}

\section{Proof of Proposition~\ref{p.main}} 
\label{s.proof}

\begin{proof}[Proof of Proposition~\ref{p.main}]
\eqref{degen}$\Rightarrow$\eqref{lag}:  The main point here is that, if $A\to B$ is a maximal degeneration, and, if we identify $V$
with $\coh_1(A_b,\mathbb{Q})$ for some $b\in B$, then 
the monodromy logarithm $N$ is an $K$-linear endomorphism of $V$ lying in the Lie algebra $\su_H$ of $\SU_H$. 
As such, we have $H(Nx,y) + H(x,Ny)=0$ for all $x,y\in V$.
(This follows from the fact that $e^{tN}$ lies in $\SU_H$ for $t\in\mathbb{R}$.) 
So set $W=N(V)$. 
By assumption, $N$ has rank $n$ as a $K$-linear endomorphism of $V$.
Moreover, as I've already pointed out, $N^2=0$ by the Monodromy Theorem~\cite[Theorem 6.1]{schmid73}.
So, for $Nx, Ny\in W$, we have 
\begin{align*}
  H(Nx,Ny) &= H(Nx,Ny) - 0\\ 
           &= H(Nx,Ny) - [H(Nx,Ny)+H(x,N^2y)] \\
           &= -H(x,N^2y)=0.
\end{align*}
\medskip

\eqref{lag}$\Rightarrow$\eqref{disc}: 
Suppose $W\subseteq V$ is a lagrangian subspace and pick a nonzero $w\in W$. 
By Lemma~\ref{l.iso}, we have $V=X\perp Y$, where $X$ is a non-degenerate hermitian
subspace of $V$ isometric to $\mathbb{H}$ with $w\in X$.

We then have $W\cap Y=W\cap X^{\perp} = W\cap v^{\perp}$, where $v$ is any vector in $X$
not in $Kw$.
So $\dim Y\cap W=\dim W-1$. 
So, since $H_{|Y}$ is non-degenerate, $Y$ satisfies the conditions of \eqref{lag}.
Thus, by induction, the discriminant of $H_{|Y}$ is $(-1)^{n-1}$. 
On the other hand, $X\cong \mathbb{H}$, so $\disc X=-1$. 
So $\disc V = (\disc X)(\disc Y) = (-1)(-1)^{n-1} = (-1)^n$ as desired.

Similarly, the signature of $X$ is $(1,1)$. 
On the other hand, by induction on the dimension, the signature of $Y$ 
is $(n-1,n-1)$. 
So, since $V=X\perp Y$ it follows easily that the signature of $V$ is $(n,n)$.
\medskip

\eqref{disc}$\Rightarrow$ \eqref{lag}: By Landherr's theorem
(Theorem~\ref{t.landherr}), the only invariants of Hermitian forms over $K$
are the dimension, the signature and the discriminant. 
So, suppose $H$ has discriminant $(-1)^n$ on $V$ and signature $(n,n)$ .
Then $H$ has the same discriminant, dimension and signature as the 
hermitian form $\varphi$ on $K^{2n}$ with standard basis

\eqref{disc}$\Rightarrow$\eqref{degen}: 
Suppose $\disc H = (-1)^n$.  
Then, by Lemma~\ref{l.class}, $H$ is isometric to $\mathbb{H}^n$.
Pick an integer $N\geq 3$, and let 
Let $X = X_0(N)$ denote the modular curve parametrizing
elliptic curves with full level-$N$ structure.
Since $N\geq 3$, there is a universal family $E\to X$.
Moreover, the family has maximal unipotent degeneration.
\end{proof}

\section{Conclusion}
\label{s.conclusion} By Proposition~\ref{p.main}, there is no maximally degenerating family of Weil type abelian varieties
of dimension $2n$ when the discriminant of the hermitian form $H$ is not $(-1)^n$. 
So, taking $n=2$, we see that, to get a maximally degenerating family of Weil type abelian $4$-folds, the 
discriminant has to be $1$. 
Then, as I said above, since Markman has proved the Hodge conjecture for Weil type cycles in dimension $4$ and 
discriminant $1$, the Kontsevich method (as explained by Zharkov in~\cite{zharkov2020tropical}) can't work to disprove the Hodge conjecture for dimension $4$ Weil type 
abelian varieties with maximal degeneration. 

On the other hand, if you extrapolate from Proposition~\ref{p.main} using Lemma~\ref{l.class}, you can guess that it is 
always possible (even when the discriminant is $\neq 1$) to get a degenerating family of $2n$-dimensional Weil type
abelian varieties with associated tropical abelian variety of dimension $n-1$. 
Perhaps something useful could be done with that family. 

\bibliographystyle{plain}

\begin{thebibliography}{10}

\bibitem{borarg}
Armand Borel.
\newblock {\em Introduction aux groupes arithm\'{e}tiques}.
\newblock Publications de l'Institut de Math\'{e}matique de l'Universit\'{e} de Strasbourg, XV. Actualit\'{e}s Scientifiques et Industrielles, No. 1341. Hermann, Paris, 1969.

\bibitem{deligneVarietesShimuraInterpretation1979}
Pierre Deligne.
\newblock Variétés de {Shimura}: interprétation modulaire, et techniques de construction de modèles canoniques.
\newblock In {\em Automorphic forms, representations and {L}-functions}, Proc. {Sympos}. {Pure} {Math}., {XXXIII}, pages 247--289. Amer. Math. Soc., Providence, R.I., 1979.

\bibitem{jacobson40hermitian}
N.~Jacobson.
\newblock A note on hermitian forms.
\newblock {\em Bull. Amer. Math. Soc.}, 46:264--268, 1940.

\bibitem{landherr1935formen}
Walther Landherr.
\newblock \"{A}quivalenz {H}ermitescher {F}ormen \"{u}ber einem beliebigen algebraischen {Z}ahlk\"{o}rper.
\newblock {\em Abh. Math. Sem. Univ. Hamburg}, 11(1):245--248, 1935.

\bibitem{lewis82isometry}
D.~W. Lewis.
\newblock The isometry classification of {H}ermitian forms over division algebras.
\newblock {\em Linear Algebra Appl.}, 43:245--272, 1982.

\bibitem{markmanweilhodge}
Eyal Markman.
\newblock The monodromy of generalized {K}ummer varieties and algebraic cycles on their intermediate {J}acobians.
\newblock {\em J. Eur. Math. Soc. (JEMS)}, 25(1):231--321, 2023.

\bibitem{markmanCyclesAbelian2nfolds2025}
Eyal Markman.
\newblock Cycles on abelian 2n-folds of {Weil} type from secant sheaves on abelian n-folds, February 2025.
\newblock arXiv:2502.03415 [math].

\bibitem{schmid73}
Wilfried Schmid.
\newblock Variation of {H}odge structure: the singularities of the period mapping.
\newblock {\em Invent. Math.}, 22:211--319, 1973.

\bibitem{vanGeemen94hodge}
Bert van Geemen.
\newblock An introduction to the {H}odge conjecture for abelian varieties.
\newblock In {\em Algebraic cycles and {H}odge theory ({T}orino, 1993)}, volume 1594 of {\em Lecture Notes in Math.}, pages 233--252. Springer, Berlin, 1994.

\bibitem{vangeemen2022fourfolds}
Bert van Geemen.
\newblock Fourfolds of {W}eil type and the spinor map.
\newblock {\em Expo. Math.}, 41(2):418--447, 2023.

\bibitem{zharkov2020tropical}
Ilia Zharkov.
\newblock Tropical abelian varieties, {W}eil classes and the {H}odge conjecture, 2020.
\newblock arxiv:2002.02347 [math].

\end{thebibliography}
\def\cprime{$'$} \def\noopsort#1{}

\end{document}